\documentclass[a4paper,12p]{article}
\usepackage[english]{babel}
\usepackage[utf8]{inputenc}
\usepackage[T1]{fontenc}
\usepackage[fleqn]{amsmath}
\usepackage{amssymb}
\usepackage{amsthm}
\usepackage{amsfonts}
\usepackage{ae}
\usepackage{enumerate}
\usepackage{stmaryrd}
\usepackage{fullpage}
\usepackage{color}
\usepackage{multirow}
\usepackage{tikz}
\usetikzlibrary{arrows}
\usepackage{enumitem}
\usepackage{comment}
\usepackage{hyperref}
\usepackage{bigcenter}
\usepackage{array}% http://ctan.org/pkg/array
\everymath{\displaystyle} % C'est quand m\UTF{00C3}\UTF{00AA}me plus joli.

\newtheorem{theorem}{Theorem}
\newtheorem{lemma}[theorem]{Lemma}
\newtheorem{corollary}[theorem]{Corollary}
\newtheorem{prop}[theorem]{Proposition}

\newtheorem{conjecture}{Conjecture}

\theoremstyle{definition}
\newtheorem{definition}{Definition}

\newtheorem{problem}{Open Problem}

\newcommand\Wider[2][4em]{%
\makebox[\linewidth][c]{%
  \begin{minipage}{\dimexpr\textwidth+#1\relax}
  \raggedright#2
  \end{minipage}%
  }%
}

\newcommand\cut{\textsc{cut}} % Affiche les cuts.
\newcommand\grundy[1]{\mathcal{G}(#1)} % Pour G(n)
\newcommand{\preG}[1]{\ensuremath{G(#1)}}
\newcommand{\G}{\ensuremath{\mathcal G}}

\newcommand{\N}{\mathbb N}
\newcommand{\C}{\ensuremath{\mathcal C}}
\renewcommand{\S}{\mathcal S}
\newcommand\AP[1]{\text{AP}(#1)}

\DeclareMathOperator{\mex}{mex}

\title{Partition games}
\author{Antoine Dailly$^{\dagger\ddagger}$\footnotemark[1] , \'Eric Duch\^ene$^\dagger$\footnotemark[1] , Urban Larsson$^\P$, Gabrielle Paris$^\dagger$\footnotemark[1]}
\date{$^\dagger$Univ Lyon, Universit\'e Lyon 1, LIRIS UMR CNRS 5205, F-69621, Lyon, France.\\$^\ddagger$Instituto de Matem\'aticas, UNAM Juriquilla, 76230 Quer\'etaro, Mexico.\\$^\P$National University of Singapore, Singapore.} %HAHA C'EST SALE MAIS \UTF{00C7}A MARCHE

\begin{document}
\renewcommand{\thefootnote}{\fnsymbol{footnote}}
\footnotetext[1]{Supported by the ANR-14-CE25-0006 project of the French National Research Agency}

\maketitle
\renewcommand{\thefootnote}{\arabic{footnote}}
\begin{abstract}
We introduce {\sc cut}, the class of 2-player partition games. These are {\sc nim} type games, played on a finite number of heaps of beans. The rules are  given by a set of positive integers, which specifies the number of allowed splits a player can perform on a single heap. In normal play, the player with the last move wins, and the famous Sprague-Grundy theory provides a solution. We prove that several rulesets have a periodic or an arithmetic periodic Sprague-Grundy sequence (i.e. they can be partitioned into a finite number of arithmetic progressions of the same common difference). This is achieved directly for some infinite classes of games, and moreover we develop a computational testing condition, demonstrated to solve a variety of additional games. 
Similar results have previously appeared for various classes of games of {\sc take-and-break}, for example {\sc octal} and {\sc hexadecimal}; see e.g. Winning Ways by Berlekamp, Conway and Guy (1982). In this context, our contribution consists of a systematic study of the subclass `break-without-take'. %are classical combinatorial games played on heaps of tokens, where the players remove tokens from a single heap and/or split it into smaller heaps. {\sc subtraction}, {\sc octal} and {\sc hexadecimal} are well-known families of such games. In this context, {\sc cut} corresponds to {\sc take-and-break} where `taking', \emph{i.e.} removal of tokens, is not allowed. A major research topic in the area concerns the possibility of regularities in their Sprague-Grundy sequences. Following this tradition, we provide a computational testing condition, which verifies if the Sprague-Grundy sequence of a given partition game is purely arithmetic-periodic . In addition, the behavior of the Sprague-Grundy sequence is explicitly given for several particular instances of $\mathcal C$ (e.g. when $1\notin \mathcal C$ or when $\mathcal C$ contains only odd values). Despite the simplicity of its ruleset, the behavior of the ruleset $\mathcal C=\{1,2\}$ remains an open problem.
\end{abstract}

\section{Introduction}
This work concerns 2-player combinatorial games related to the classical game of {\sc nim}, but instead of removing objects, say beans, from heaps, players are requested to partition the existing heaps into smaller heaps, while the total number of beans remain constant. We prove several regularity results on the solutions of such games. 

Integer partition theory, related to Ferrer diagrams and Young tableaus, is a classical subject in number theory and combinatorics, dating back to giants such as Lagrange, Goldbach and Euler; it concerns the number of ways you can write a given positive integer as a sum of specified parts. In most generality, to each positive integer $n\in \N=\{1,2,\ldots\}$, there belongs a number $p(n)$, which counts the unrestricted number of ways this can be done. For example $4=3+1=2+2=2+1+1=1+1+1+1$, so $p(4)=5$. We may index this partition number by saying exactly how many parts is required, and write $p_k(n)$ for the number of partitions of $n$ in exactly $k$ parts. Thus, in our example, $p_2(4)=2$ and $p_3(4)=1$. We could also define $p_{k, c}=p_k+p_ c$ and so on. The number of partitions can be beautifully expressed via generating functions, where recurrence formulas, congruence relations, and several asymptotic estimates are known, proved more recently by famous number theorists such as Ramanujan, Hardy, Rademacher and Erd\H{o}s in the early 1900s. About the same time, a theory of combinatorial games was emerging, via contributions by Bouton, Sprague and Grundy and others, seemingly unrelated to the full blossom of number theory.

An \emph{integer partition game} can be defined by 2 players alternating turns and by specifying the legal partitions, say into exactly 2 or 3 parts, until the current player cannot find a legal partition of parts, and loses. Thus, from position 4, then $3+1,2+2,2+1+1$ are the legal move options---if you play to $2+2$ you win, and otherwise not. It turns out that the idea for how to win such games is coded in a `game function', discovered independently by the mathematicians Sprague and Grundy, which, by the way, does not appear to have any direct relation to the partition function. For example, the partition functions are nondecreasing, but if a Sprague-Grundy function is nondecreasing the game is usually rather trivial, such as the game of {\sc nim} on one heap. 

Let us begin by giving the relevant game theory background to our results, that several \emph{partition games} have either a periodic or an \emph{arithmetic-periodic} Sprague-Grundy sequence.\footnote{An arithmetic-periodic sequence can be partitioned into a finite set of arithmetic progressions with the same common difference.} We call the class of partition games  {\sc cut}.

An impartial combinatorial game $G=G(\mathcal R, X, x)$ is given by a ruleset $\mathcal R$, a set of positions $X$, and a starting position $x\in X$; the ruleset $\mathcal R$ specifies how to move from any (starting) position. Two players alternate in moving, with one of the players assigned as a starting player, and in normal play a player who cannot move loses. An {\em option} of $G$ is a game $G'=G'(\mathcal R, X, y)$ that can be reached in one move, and all games have finitely many options, and finite rank.

The {\em Sprague-Grundy value} \cite{S, G} (or \G-value for short) of an impartial normal play game $G$ is the nonnegative integer $$ \grundy{G}= \mex \{\grundy{G'} \mid G' \text{ is an option of } G \}, $$ where $\mex(U)=\min \N_0\setminus U$ is the smallest nonnegative integer that does not belong to the strict subset $U\subset\N_0=\N\cup\{0\}$.

The Sprague-Grundy value can be used to determine the winner in perfect play. Indeed, a game $G$ satisfies $\grundy{G}=0$ if and only if playing first in $G$ loses; \emph{i.e.} the \emph{previous player} wins.

Heap games are typically played on several heaps of the form $H_n$, where $n\in \N_0$ denotes the number of beans in a heap, and there is a given ruleset that specifies the legal options on the heaps. An important concept is that of a disjunctive sum of heaps.
A \emph{disjunctive sum} of $k$ heaps of sizes $i_0,\ldots , i_k$ is denoted $H=H_{i_0}+\cdots +H_{i_k} $, where by moving, a player chooses one of the heaps, say $H_{i_j}$ and makes a move in this heap according to a given ruleset. The other heaps remain the same. That is, a typical move option of the game $H$ is of the form $H'$ with
\begin{align}\label{eq:moveoption}
H' = H_{i_0} + \cdots +H_{i_{j-1}}+{H_{i_j}}'+H_{i_{j+1}}+\cdots+H_{i_k},
\end{align}
where ${H_{i_j}}'$ is a move option on the heap $H_{i_j}$.\footnote{In general, an option may include several heaps, where each new heap is smaller than the previous heap, to allow recursive computation of the Sprague-Grundy values; in this paper, exactly one of the heaps in a sum of heaps will be cut (broken, partitioned, split) into at least 2 heaps. See also Section~\ref{sec:background} for a surrounding context and a review of classes of heap games, such as octal games, hexadecimal games, and so on.}

By the main result of Sprague-Grundy Theory, $\grundy{H}=\grundy{H_{i_0}}\oplus \cdots \oplus \grundy{H_{i_k}}$, where $\oplus$ is the standard \emph{nim-sum} operator.\footnote{Consider nonnegative integers $m=\sum e_i2^i$ and $n=\sum f_i2^i$. Then $n\oplus m=\sum (e_i\text{XOR} f_i) 2^i$.}  Therefore, the previous player wins if and only if the nim-sum is 0. 

The following result about the nim-sum operator will be used several times in this paper.

\begin{lemma}\label{lem:parity}
  Consider any $a_0,\dots, a_m\in\N_0$. Then $$a_0\oplus a_1 \oplus  \dots \oplus a_m \equiv (a_0+\cdots + a_m) \bmod 2$$
  and
  $$ a_0+\dots +a_m\geq a_0\oplus \dots\oplus a_m.$$
\end{lemma}
\begin{proof}
Both the nim-sum and the standard sum is odd if and only if the number of odd $a_i$s is odd. This proves the first part.
 For each index $j\in[0,m]$, set $\sum_i e_i(j)2^i=a_j$. Then, by definition, for each $i$, $\sum_j e_j\ge\bigoplus e_j$, which proves the second part.
\end{proof}

As evidenced by (\ref{eq:moveoption}), to study a heap game played on a disjunctive sum of heaps, it suffices to describe the rules on a single heap. 

\begin{definition}[\cut]\label{def:cuts}
The \cut\ \emph{ruleset} $\C\subseteq \N$ specifies the allowed number of \emph{cuts} (or \emph{splits}) of a given heap $H_n = H_n(\C)$ into a disjunctive sum of heaps, as $$H_n=\{H_{n_0}+\cdots+H_{n_ c} \mid  c \in \C, \forall i: n_i>0, n_0+\cdots+n_ c=n \}$$
An instance of $H_n$ is called a $c$-\emph{cut} or an \emph{option} of $n$, and is denoted $O_n\in H_n$.
\end{definition} 

That is, by combining Definition~\ref{def:cuts} with \eqref{eq:moveoption}, a move consists in splitting one of the heaps in the disjunctive sum $H$ into $ c+1$ non-empty heaps with $ c\in \mathcal C$, while the other heaps remain the same.\footnote{Here $\C$ is an \emph{invariant ruleset} in an analogous sense as defined in \cite{DR} for subtraction games. The term `invariant' is with respect  to the size of the heap, so that if $c\in\C$ then a $c$-cut is available for all $H_n$, with $n>c$. We often think of the sequence of all heap games for a given ruleset $\C$, as  the `fingerprint' of a game. As it is common in CGT papers, we will deemphazise the actual strategies of how to win individual games, and instead evaluate a game in terms of the sequence of Sprague-Grundy values of all heaps.}

Let us illustrate Definition~\ref{def:cuts} with an example. Assume that $\C=\{3,4\}$, that is, we can split a heap into either~4 or~5 non-empty heaps. Assume that the initial position is the heap $H_{13}$. The first player decides to split it into four heaps, leaving $H_1+H_3+H_3+H_6$. The second player has to choose one of those four heaps and split it, which they can only do with $H_6$, after which the game ends since no heap can be split anymore. Note that the first player could have chosen to split the heap into four equal heaps of size~3 and one of size~1, which would have ended the game immediately.

In this paper, we find the $\mathcal{G}$-sequences for several instances of {\sc cut}. In Section~\ref{sec:psg}, we solve the classes:

\begin{itemize}
\item $\C\subseteq$ `odd numbers';
\item $\min\C\ge 2$;
\item $\{1,2,3\}\subset \C$;
\item $\C=\{1,3,2k\}$, $k$ even.
\end{itemize}

Later, in Section~\ref{sec:aptsg}, we solve several more instances of {\sc cut} by adapting a computational testing condition, that we will call the AP-test, summarized in Table~\ref{tab:someGames} (see also Table~\ref{tab:recap}).

A typical representative of {\sc cut} is $G(\C,\N,H)$, but we usually omit $\N$ since it is understood, and also $H$, because we will gain knowledge of the Sprague-Grundy values of individual heaps (which then implies knowledge about a disjunctive sums of heaps). In this spirit, we continue to write just $G(\C)$, and we will return to this notation.

To simplify notation, for the Sprague-Grundy value of a heap $H_n$ of size $n\ge 1$, we will write $\grundy{H_{n}} = \grundy{n}$, and for a disjunctive sum of heaps, we write $H_{n_0}+\cdots +H_{n_k}=(n_0,\ldots , n_k)$ and $\grundy{H_{n_0}+\cdots +H_{n_k}}=\grundy{n_0,\ldots , n_k}$.

There are several similarities with {\sc nim}-type games, where only removal is possible. For example, if $\mathcal C = \N$, then $H_{n+1}$ is equivalent with a {\sc nim} heap of size $n$.  This follows, because both are normal play impartial games and $H_n(\textsc{nim})\subset H_{n+1}(\mathcal C)$, and argue by induction on the Sprague-Grundy values. See also Proposition~\ref{prop:equivalence} in Section~\ref{sec:background} for a more general observation relating to  the classical {\sc take-and-break} games.

This work was much inspired by in particular two classical combinatorial games, namely {\sc grundy's game}~\cite{G} and {\sc couples-are-forever}~\cite{couples}; see also \cite{F}. In the first one, a move consists in choosing a heap and splitting it into two heaps of \emph{different} size. The latter one allows to split any heap of size \emph{at least three} into two heaps. For both games, some extra constraints have been adjoined to the type of possible splits, and no regularities in the $\mathcal{G}$-sequences have yet been observed. Here, we study partition games with no extra constraint than prescribed \emph{splitting}- or \emph{cut}-numbers.

{\sc nim}-type games on finite rulesets are known as {\sc subtraction} \cite{WW}, and it is a folklore result that all such games are periodic (i.e. their Sprague-Grundy sequences $(\grundy{H_n})_{n\in\N}$ are periodic). This follows by a simple combinatorial counting argument. For the game \cut, the situation is more varied, and we will encounter both periodic, and arithmetic-periodic games. Some rulesets are not yet fully understood, and in particular the ruleset $\C=\{1,2\}$ remains a mystery (it does not seem to be arithmetic-periodic); see Section~\ref{sec:conper}.

We use the notion of a `game' in at least two different ways, both depending on the notion of a `ruleset'. A $\cut$ game that can be enjoyed as a recreational game is a ruleset $\C$ together with a heap $H_n$ of a given size $n$ (or a disjunctive sum of heaps $H$). On the other hand, a minimal requirement for a strategic understanding of a ruleset is to acquire its \G-sequence (\emph{i.e.} $\G(H_1),\G(H_2),\ldots$). One consequence of the \G-sequence is that you might be indifferent between two rulesets (independently of the size of a heap). From this perspective it is more natural to think of a game, as its sequence of \G-values. The context will decide, and we believe that the word ``game" should remain a word without a precise definition, available for use in various contexts; on the other hand, a ``ruleset" should always have a well defined meaning.

In this spirit, and to emphasize the motivation of this paper, we will identify $\preG{\C}$ with its sequence of Sprague-Grundy values $\G(H_1),\G(H_2),\ldots$. For example, we write $\preG{\{1\}}=0,1,0,1,\ldots $ (See Proposition~\ref{prop:1odd}.)

We will index the elements in a given set $\C$ such that, for all $i$, $ c_i <  c_{i+1}$. 

A well-studied periodicity problem, by Richard Guy, on so-called octal games,\footnote{Such games allow at must one cut, but it might be combined with various removals; see Section~\ref{sec:background} on class of {\sc take-and-break}.} does not transfer to the full class of partition games. Indeed, the following lemma establishes that if $\C$ contains an even cut-number, the Sprague-Grundy values are unbounded.

\begin{theorem}
\label{lem:ericslemma}
If a ruleset $\C$ contains an even cut-number, then $G(\C)$ (\emph{i.e.} its Sprague-Grundy-sequence) is unbounded. In particular, if the smallest even cut-number is $c\in\C$, then any arithmetic progression of the form $x, x+c, x+2c,\ldots $, contains no repetition of \G-values, and hence it contains infinitely many \G-values.
\end{theorem}

\begin{proof}
It suffices to show that, for every pair of heaps $H_{x_1}, H_{x_2}$, with $x_1\neq x_2$ such that $\grundy{x_1}=\grundy{x_2}$, then $x_1 \not\equiv x_2 \bmod c$.

Suppose that $x_1 \equiv x_2 \bmod c$. Then $x_1 = q_1c+r$ and $x_2 = q_2c+r$ for some $0< r\le c$, $q_1, q_2\in\N_0$. Assume without loss of generality that $x_1>x_2$. Thus $q_1-q_2\ge 1$, and therefore one can $c$-cut $H_{x_1}$ to the option $O_{x_1}=(x_2,q_1-q_2,\dots, q_1-q_2)$.  Since $c$ is even, by the definition of nim-sum, $\grundy{O_{x_1}}=\grundy{x_2}$, and thus, by definition of the $\mex$-function, $\grundy{x_1} \neq \grundy{x_2}$.
\end{proof}

In Section~\ref{sec:psg}, we consider several families of partition games (e.g. those where $1\notin \C$, or those with only odd values in $\C$) and prove their pure periodicity or pure arithmetic-periodicity. For the remaining families, many games seem to have a purely arithmetic-periodic behavior. To deal with them, we provide, in Section~\ref{sec:aptsg}, a set of testing conditions whether a game is purely arithmetic-periodic, and apply them to particular instances. In Section~\ref{sec:background}, we give some background on classical games of {\sc take-and-break}, such as {\sc octal} and {\sc hexadecimal} games. Finally, in Section~\ref{sec:conper}, we mention some remaining rulesets $\C$ for which a possible regularity of the $\mathcal{G}$-sequence remains open, as well as other open questions.

\section{Particular partition games}\label{sec:psg} 

In this section, we study some specific families of partition games with arithmetic-periodic Sprague-Grundy sequences. The notation $(m_1, \ldots, m_p)~(+s)$   describes the arithmetic-periodic sequence of \emph{period} $p$ and \emph{saltus} $s$ for which the first $p$ values are $m_1, \ldots, m_p$. We write $(m_i,\ldots,m_j)^k$, if a subsequence $(m_i,\ldots,m_j)$ is repeated $k$ times. Thus, for example,  $(0,1,2)^2~(+3)$ denotes the arithmetic-periodic sequence of period 6 and saltus 3, \emph{i.e.} $0,1,2,0,1,2,3,4,5,3,4,5,\ldots $

As in Theorem~\ref{lem:ericslemma}, we will often encounter a unique decomposition of an integer via the division algorithm.  For a given divisor (or period) $p\in \N$, for any $n\in \N$, there exists a unique ordered pair of nonnegative integers $(q_n,r_n)=(q, r)$ such that $0< r \le p$ and $n=pq+r$.\footnote{We use the convention $0< r \le p$, since $H_1$ is terminal in \cut.} For consistency we will always use the letters $q$ and $r$ with this meaning, sometimes, but not necessarily, indexed to indicate their origin. 

We first show that if $\C$ contains only odd numbers, including the possibility to split a heap into exactly two heaps, then the \G-sequence of \preG{\C} is purely periodic with period $2$.

\begin{prop}\label{prop:1odd}
Let $\C$ be a (possibly infinite) ruleset consisting exclusively of odd cut-numbers, with $\min \C=1$. Then $\preG{\C}=0,1,0,1,\ldots$, \emph{i.e.} $\grundy{n}=0$ if $n$ is odd and $\grundy{n}=1$ if $n$ is even.
\end{prop}
\begin{proof}
Note that $\grundy{1}=0$ and $\grundy{2}=1$, since $\min \C =1$.

Suppose  that $n$ is odd, and study a generic option of $H_n$. The number of heaps in $O_n$ of odd size must be odd. Therefore the number of heaps of even size is also odd. By induction, this shows that $\grundy{O_n}=1$.

Suppose next that $n$ is even. The number of heaps of odd size in $O_n$ must be even, and so the number of heaps of even size must also be even. By induction, this gives the claim for even heap sizes.
\end{proof}

In this section, and later, we get repeated use of a very simple lemma.

\begin{lemma}\label{lem:h0}
Consider a ruleset $\C$ and a disjunctive sum $h=(k,h_1, \ldots , h_c)$. If all non-zero \G-values can be paired, except for $\grundy{k}$, then $\grundy{h}=\grundy{k}$. Moreover, if $h$ is a $c$-cut of $H_n$, with $c\in \C$, then $H_n$ has an option with \G-value $\grundy{k}$.
\end{lemma}
\begin{proof}
Obvious.
\end{proof}

This result is useful to prove lower bounds on \G-values from given positions. (Especially since the proofs will tend to divide into several cases with similar arguments.) The upper bounds can be proved directly (using Lemma~\ref{lem:parity}), or one can show that the anticipated value cannot have appeared before for smaller heap sizes. In either way, we rely heavily on induction.

In fact, for the upper bounds, an extension of Lemma~\ref{lem:parity} will prove useful. It reveals a general property of a combination of the division algorithm and the nim-sum.

\begin{lemma}\label{lem:nooptq}
Let $h=(h_0,\ldots , h_c)$ be a $c$-cut of $n=qp+r$,\footnote{That is, $h$ is a $(c+1)$-partition of $n$.} and where, for all $i$, $h_i=pq_i+r_i$. If $c\ge p$, then $\bigoplus_{i=0}^m q_i \ne q$.
\end{lemma}
\begin{proof}

By way of contradiction, suppose that $q = q_0 \oplus \cdots \oplus q_m$.

Since $\sum h_i=n$, we have $$\sum_{i=0}^c (q_i c + r_i) = q c+r. $$

Note that $r\equiv \sum_{i=0}^c r_i \pmod{c}$ forces $r\le \sum_{i=0}^c r_i$, and hence $ \sum_{i=0}^c q_i \leq q $.

By combining this with the assumption and Lemma~\ref{lem:parity}, we get
$q = \bigoplus_{i=0}^c q_i \le \sum_{i=0}^c q_i \le q$,
and hence

\begin{align}\label{eq:lea}
\bigoplus_{i=0}^c q_i = \sum_{i=0}^c q_i.
\end{align}
But then $c\ge p$ gives a contradiction. Namely, by (\ref{eq:lea}) and for all $i$, $r_i>0$, we get $r= \sum_{i=0}^c r_i > c\ge p$, which contradicts the definition of $ r \le p$.
\end{proof}

First, we study partition games, where a heap splits into at least three heaps. In this case, optimal play is reduced to using only $ 2\le c_1=\min \C$, and the Sprague-Grundy sequence is purely arithmetic-periodic with period $ c_1$ and saltus 1. Here, a cut-set may be infinite. For example, if $c_1=2$, we prove that $\preG{\C}=0,0,1,1,\ldots $ Let $r$ be the smallest positive integer congruent with $n\pmod {c_1}$. We show that for all heaps $H_n$, $\grundy{n}=\frac{n-r}{c_1}=\left\lfloor \frac{n-1}{c_1}\right\rfloor$.

\begin{prop}
  \label{lem:k*}
  Consider a (possibly infinite) ruleset $\C$ with $ c_1 = \min \C \geq 2$.
  Then, $\preG{\C}=(0)^{ c_1} ~(+1)$.
\end{prop}
\begin{proof}

 We want to prove that for every positive integer $n$, $\grundy{n}=q$, where $n=c_1q+r$. Clearly $\grundy{1} = \grundy{2}=0$. Assume the statement holds for all $n'< n$.  By $c_1=\min\C$, Lemma~\ref{lem:nooptq} gives that there is no option $O_n$ with $\grundy{O_n} = q$.

It remains to  prove that $H_n$ has an option of \G-value $g$ for all $g \le q-1$.
There are two cases:
\begin{enumerate}
\item If $ c_1$ is even, then, for each integer $g \le q-1$, let $O_{n} = (g c_1+r, q-g,\ldots,q-g)$ be a $ c_1$-cut (with $c_1$ copies of $H_{q-g}$). Here Lemma~\ref{lem:h0} applies, but we give the details since it is the first occurrence of the lemma. This is an option of $H_n$, since $g c_1+r+(q-g) c_1 = q c_1 + r = n$. Furthermore, $\grundy{O_{n}} = \grundy{g c_1 + r} \oplus ( c_1 \otimes \grundy{q-g}) = g\oplus 0 = g$, since by induction $\grundy{g c_1+r} = g$ and by $ c_1$ even.
\item If $c_1$ is odd, for each \G-value $g \le q-1$, we define a $c_1$-cut $O_{n}$, such that $\grundy{O_{n}}=g$.
We have two subcases:
    \begin{enumerate}[label*=\arabic*]
    \item If $q-g-1$ is even,
let
$$\begin{array}{lllll}
h_0&=& g c_1 +r&&\\
&&&&\\
h_j&=& \frac12 (q-g-1) c_1 +1 &&\text{for } j=1,2\\
&&&&\\
h_j&=& 1&&\text{for } 3\leq j\leq  c_1
\end{array}$$

Then $O_n=(h_0,\ldots,h_{c_1})$ is an option of $H_n$  since $h_1 = h_2$ are non-negative integers, since $q-g$ is odd.  Apply Lemma~\ref{lem:h0}, using $\grundy{g c_1+r}=g$ by induction and $\grundy{1}=0$.
\item If $q-g-1$ is odd,
let
$$\begin{array}{lllll}
h_0&=&g c_1+r&&\\
&&&&\\
h_j&=& \frac12 ((q-g-1) c_1 +1)&&\text{for } j=1,2\\
&&&&\\
h_j&=& 2 &&\text{for } j=3 \\
&&&&\\
h_j&=& 1 &&\text{for } 4\leq j \leq  c_1
\end{array}
$$

Apply Lemma~\ref{lem:h0} on $O_n=(h_0,\ldots,h_{c_1})$, using $\grundy{1}=\grundy{2}=0$, and that, by induction $\grundy{g c_1+r} = g$. (If $c_1=3$, then omit the last item.)
\end{enumerate}
\end{enumerate}

This proves that there is an option with \G-value $g$ for all $0\leq g < q$, and thus $\grundy{n} \geq q$.
Together with (\ref{eq:lea}), we obtain, for all $n$, $\grundy{n}=q$.
  \end{proof}

Next, we study partition games for which the players can split a heap into two, three or four heaps. In this case, even if the players are allowed to split a heap into more than four heaps, then the \G-sequence is purely arithmetic-periodic with period 1 and saltus 1.

\begin{prop}
\label{lem:123}
Let $\C$ be a (possibly infinite) ruleset with $\{1,2,3\}\subseteq \C$. Then, for all heaps $H_n$, $\grundy{n} = n-1$.
\end{prop}

\begin{proof}
For base cases, $\grundy{1}=0$ and $\grundy{2}=1$.

Apply Lemma~\ref{lem:nooptq} with $p=1$, $r=1$, $n-1=q$. Hence there is no option of $n$ of \G-value $n-1$.

To prove that $\grundy{n} \ge n-1$, we must find an option $O_n$ with $\G(O_n)=g$, for all $g < n-1$.
The 1-cut $O_n = (1, n-1)$, satisfies $\grundy{O_n} = n-2$ by induction. Otherwise, let $g<n-2$, and apply Lemma~\ref{lem:h0}.
There are four cases depending on the parities of $n$ and $g\le n-3$. If $n$ and $g$ have the same parity, we use a 3-cut, and otherwise a 2-cut. In all cases, we may use that, by induction, $\grundy{O_n}=\grundy{g+1}=g$.

   \begin{center}

   {\renewcommand{\arraystretch}{2}%
         \begin{tabular}{|c|c|c|}
           \hline
           $O_n$ & $n$ odd & $n$ even \\
          \hline
          {$g$ odd} &$\left(g+1, 1,\frac{n-g-2}{2},\frac{n-g-2}{2}\right)$ & $\left(g+1, \frac{n-g-1}{2},\frac{n-g-1}{2}\right) $\\
           {$g$ even} &$\left(g+1,  \frac{n-g-1}{2},\frac{n-g-1}{2}\right)$ & $ \left(g+1,1, \frac{n-g-2}{2},\frac{n-g-2}{2}\right) $\\

                      \hline
         \end{tabular}}
         \end{center}

Altogether, $\grundy{n} = \mex(\{0,\ldots ,n-2\})=n-1$.
\end{proof}

Finally, we study finite partition games, namely the {\sc cut} class where a player can split a heap into 2, 4 or $2k+1$ heaps, where $k\ge 1$ is a given game parameter (note that this includes the game  \preG{1,2,3}). In this case, the Sprague-Grundy sequence is purely arithmetic-periodic with period $2k$ and saltus~2. For example, if $k=2$, then $\preG{1,3,4}=0,1,0,1,2,3,2,3,\ldots$ In general, we prove that $\grundy{n}=2\left\lfloor\frac{n}{2k}\right\rfloor+1-(n\bmod 2)$.

\begin{prop}
\label{lem:132k}
For any given $k \in \N$, let $\C=\{1,3,2k\}$. Then, $\preG{\C}=(0,1)^{k} ~(+2)$.
\end{prop}

\begin{proof}
Given $k \in \N$, we must show that the period is $p=2k$. Hence, for any heap $H_n$, define integers $q,r$ by $n = 2kq+r$, $0 < r\le 2k$.

More precisely, we prove that, for all heaps $H_n$, $\grundy{n} = 2q +1 -(r\bmod 2)$.

Note that (for any $k$), $\grundy{1}=0$ ($q=0$ and $r=1$) and $\grundy{2}=1$ ($q=0, r=2$). Consider $n\geq 3$.

%%%%%%%%%%%%%Here Lemma applies%%%%%%%%%%%%%%%%%%%

We show first that $\grundy{n}\le 2q+1-(r\bmod{2})$.

If $\grundy{n} > 2q+1-r\pmod{2}$, then $H_n$ must have an option $O_n$ with $c\in \C$ such that $\grundy{O_n}=2q+1-(r\bmod 2)$.

In case $c=2k$, assume that there is a $c$-cut $O_n=(n_{q_0}, \ldots, n_{q_c})$, with
$$ 2kq+r=n=\sum_{i=0}^c (2kq_i+r_i)=2k\sum_{i=0}^c q_i +\sum_{i=0}^c r_i .$$
Then Lemma~\ref{lem:nooptq} applies with $p = 2k$, so  $q \ne q_0 \oplus \cdots \oplus q_c$. Now,

\begin{align}
2q+(1-r\bmod 2) &= \grundy{O_n}\\
&=\bigoplus_{i=0}^c\grundy{n_{q_i}}\label{eq:middle}\\
&=\bigoplus_{i=0}^c (2q_i+(1-r_i\bmod 2))\\
&= 2\bigoplus_{i=0}^c q_i + \bigoplus_{i=0}^c(1-r_i\bmod2),\label{eq:last}
\end{align}

The  equality (\ref{eq:middle}) holds by induction and (\ref{eq:last}) holds since 2 is a power of 2, and since, for all $i$, $(r_i\bmod 2)<2$. By combining  $q \ne q_0 \oplus \cdots \oplus q_c$ with $r\bmod 2=\bigoplus_{i=0}^cr_i\bmod2$, this gives a contradiction. 

Therefore assume $c\in \{1,3\}$. Then, by (\ref{eq:last})
\begin{align}
1-r \bmod 2 &= \bigoplus_{i=0}^c (1-r_i \bmod 2)\\
&= \left(\bigoplus_{i=0}^c r_i \bmod 2\right)\label{eq:modd}\\
&= r\bmod 2 \label{eq:modb}
\end{align}

Equality (\ref{eq:modd}) holds since $c$ is odd, and (\ref{eq:modb}) holds by $\sum_{i=0}^c r_i  = r$ , a contradiction.

Thus, there is no option of $H_n$ with \G-value $2q+(1-r\bmod{2})$. 

We now prove that, from a heap $H_n$ of $n=2kq+r$ counters, there is an option of any \G-value smaller than $2q+(1-r \bmod 2)$.  Lemma~\ref{lem:h0} applies, but we give full details since the desired options vary a bit depending on various parities. There are two cases, depending on the parity of $r$.

\begin{enumerate}
\item If $r$ is odd, then $2q+(1-r\bmod{2}) = 2q$, and
\begin{enumerate}[label*=\arabic*]
\item for each integer $x\in [0,q-1]$,
$$ O_{n}=(2kx+r,q-x,\ldots,q-x)$$
is an option of $H_n$, obtained by a $2k$-cut. By induction, $\grundy{O_{n}}=2x$, which gives the even \G-values in $[0, 2q-2]$.
\item if $r=1$, for each integer $x\in [1,q-1]$,
$$ O_n =(2kx, 1, (q-x)k, (q-x)k)$$
is an option of $H_n$, obtained by a $3$-cut. By induction, $\grundy{O_n}=2x-1$, which gives the odd \G-values in $[1,2q-3]$, and the value $2q-1$ is obtained by the option $O_n = (2kq,1)$.
\item if $r\ge 3$, for each integer $x\in [0,q-1]$,
$$ O_n =(2kx+r-1, 1, (q-x)k, (q-x)k)$$
is an option of $H_n$, obtained by a $3$-cut. By induction, $\grundy{O_n}=2x+(r\bmod 2)=2x+1$ since $r$ is odd, which gives the odd \G-values in $[1,2q-1]$.
\end{enumerate}
Altogether, this implies $\grundy{n}\geq 2q$, if $r$ is odd.
\item If $r$ is even, then $2q+(1-r\bmod{2})=2q+1$, and
\begin{enumerate}[label*=\arabic*]
\item for each integer $x\in [0,q-1]$,
$$ O_{n}=(2kx+r+1,q-x,\ldots,q-x)$$
is an option of $H_n$, obtained by a $2k$-cut. By induction, $\grundy{O_{n}}=2x+1$, which gives the odd \G-values in $[1,2q-1]$.
\item for each integer $x\in [0,q-1]$,
$$ O_n =(2kx+r, 1, (q-x)k, (q-x)k)$$
is an option of $H_n$ obtained by a $3$-cut. By induction, $\grundy{O_n} = 2x$, which gives the even \G-values in $[0,2q-2]$, and the value $2q$ is obtained by the option $O_n = (2kq+r,1)$.
\end{enumerate}
Thus $\grundy{n}\geq 2q+1$, for even $r$.
\end{enumerate}
Hence, for all heaps $H_n$, $\grundy n = 2q+1-(r\bmod{2})$.
\end{proof}

Note that when $k=1$, Proposition~\ref{lem:132k} gives the same result as Proposition~\ref{lem:123} when $k = 3$ (and as such, $\C=\{1,2,3\}$).

The above results cover a large range of partition games, but in remaining cases we were not able to have direct proofs. Yet, many of them seem to be well-behaved. The next section is devoted to building a test that allows us to prove (using a small number of computations) that a given game is purely arithmetic-periodic. We then use this test to prove that some games have a purely arithmetic-periodic Sprague-Grundy sequence.

\section{An arithmetic-periodicity test for partition games}\label{sec:aptsg}

The purpose of this section is to provide 
an explicit tool to verify if a partitioning game is purely arithmetic-periodic by computing a small number of initial \G-values. Similar results are known as the {\sc subtraction}, {\sc octal} and {\sc hexadecimal} periodicity tests (see Section~\ref{sec:regtbg}, Theorem~\ref{thm:octal} and~\cite{Hexa}), in the latter case concerning pure arithmetic-periodicity. Recall that for {\sc octal}, the number of computations to prove the periodicity is in the range of twice the period, whilst it takes at least 7 times the period to prove the arithmetic-periodicity of {\sc hexadecimal} (together with a couple of additional tests). For (finite) {\sc subtraction}, the range of computations is given by the sum of the period and the highest number in the subtraction set. We review this development in Section~\ref{sec:background}.

In Section~\ref{sec:APT}, we prove that computing at most the first $4p$ values of the $\mathcal{G}$-sequence (where $p$ is the expected period, which should be determined by a blind computation) is enough to prove pure arithmetic-periodicity. We will also show that in some cases (depending on $\C$), the first $3p$ values are even sufficient (Section~\ref{sec:relaxedAP}).

\subsection{The AP-test}\label{sec:APT}

In this section, we describe a test that will be used to verify if a given partition game, on a finite ruleset, is purely arithmetic-periodic.

\begin{definition}[Arithmetic-Periodic Test]\label{def:APT}
  Consider a partition game  on a finite ruleset $\C$, and a single heap $H_n$, with $n\in \N$. Suppose there exists a smallest positive integer $p$ and a power of two $s=2^t$, $t\in \N_0$, with $1\le s\le p$, such that
  \begin{enumerate}[label=AP\arabic*:]
  \item if $n\leq 3p$ then $\grundy{n+p}=\grundy{n}+s$,
  \item $\{\mathcal{G}(n)\mid n\le p\}=\{0,1,\ldots, s-1\}$, and
  \item if $n\in[3p+1,4p]$ then $\forall g\in[0,s-1]$ $\exists$ $c$-cut $O_n$ with $2\le c\in \C$ such that $\grundy{O_n} = g$.
  \end{enumerate}
  Then the Arithmetic-Periodicity Test is $(p,t)$-verified for ruleset $\C$, and we write $\AP{\C}=(p,t)$.
\end{definition}
The first two conditions are rather standard to prove the periodicity of {\sc take-and-break} games; similar conditions are required in the {\sc subtraction} periodicity, {\sc octal} periodicity and {\sc hexadecimal} arithmetic-periodicity tests (see Section~\ref{sec:background} for more discussion on this). However, contrary to those (when applicable), we require the saltus to be a power of two in order to prove arithmetic-periodicity. The third condition seems more unusual, and is used for the base case of the proof. We will see in the next subsection that, for some rulesets, the third condition AP3 can be directly deduced from AP1 and AP2, which suggests that there could be a periodicity test closer to the existing tests for subtraction, octal and hexadecimal games. We now state the main result of this section.
\begin{theorem}\label{lem:BigFatLemma}
Suppose that $\AP{\C}=(p,t)$, where $\max\C\le 4p$. 
Then \preG{\C} is arithmetic-periodic with period $p$ and saltus $2^t$, \emph{i.e.} for all heaps $H_n$, $\grundy{n+p}=\grundy{n}+2^t$.
\end{theorem}
In other words, if a partitioning game verifies the AP-test, then it is purely arithmetic-periodic. Note that both in the definition of the AP-test as well as in the main theorem, the saltus of the \G-sequence is a power of 2. The bound $\max\C\le 4p$ ensures that AP3 applies as a base case for induction.

\begin{comment}
\begin{lemma}\label{lem:tight-parity}
  Let $n$ be a positive integer.  For all $0\leq l\leq n$ with $l\equiv n \bmod 2$ there exist three nonnegative integers $a_0$, $a_1$ and $a_2$ such that $a_0 \oplus a_1\oplus a_2=l$ and $a_0+a_1+a_2=n$.
\end{lemma}
\begin{proof}
  Let $n$ be a positive integer and $l\geq n$ be a non negative integer of same parity as $n$.\\
Let   $a_0=l$, $a_1=a_2=\frac{n-l}{2}$. We have $a_0+a_1+a_2= l+(n-l)=n$ and $a_0\oplus a_1 \oplus a_2 = l$ since $a_1\oplus a_2=0$.
\end{proof}
\end{comment}

In order to prove this result, we use a couple of lemmas. As before, we will frequently make use of the fact that for every heap $H_n$, and a given period $p$, there exists a unique pair $(q,r)=(q_n,r_n)$ such that $n=pq +r$ with $q\in \N_0$ and $0< r\le p$. The next result gives the closed formula corresponding to a purely arithmetic-periodic \G-sequence. (It applies to {\sc cut}, but also any other ruleset.)

\begin{lemma}\label{lem:method}
  Consider any heap game. Suppose there exists a period $p$, a saltus $s$, with, for all $n > p$, $\grundy{n}=\grundy{n-p}+s$.
  Then, for all $n>0$, $\grundy{n}=sq_n+\grundy{r_n}$.
\end{lemma}
\begin{proof}
  For all $1\leq n\leq p$, $n=r$ with $q=0$, and hence, as a base case, $\grundy{n}=\grundy{r}$.

  Let $n=pq+r > p$. 
  The \G-value of $H_n$ is
  \begin{align}
  \grundy{n}&=\grundy{n-p}+s\notag\\
  &=\grundy{p(q-1)+r}+s\notag\\
  &=s(q-1)+\grundy{r}+s\label{eq:ind}\\
  &=sq+\grundy{r},\notag
  \end{align}
  where \eqref{eq:ind} is by induction.
\end{proof}
One nice consequence of Lemma~\ref{lem:method} is that the \G-values decompose, provided that the saltus is a power of 2, and this simple result outlines our approach. 
\begin{lemma}\label{cor:prop14}
Consider a ruleset $\C$ and a disjunctive sum of heaps $H = (h_0,\dots, h_c)$.
Suppose that Lemma~\ref{lem:method} is satisfied, with $h_i\leq n$, for all $0\leq i\leq c$, and with the two additional constraints
\begin{itemize}
\item $s=2^t$, $t\in \N_0$,
\item $\grundy{n}<s$ for all $1\leq n \leq p$.
\end{itemize}
Then
\begin{equation}
\grundy{H}=(q_0\oplus\dots\oplus q_c)s+\grundy{r_0}\oplus\dots\oplus\grundy{r_c},
\label{etoile}
\end{equation}
where for all $i$, $(q_i,r_i)=(q_{h_i},r_{h_i})$.
\end{lemma}
\begin{proof}
\begin{align}
\grundy{H}&=\grundy{q_0p+r_0,\ldots , q_cp+r_c}\\
&=\grundy{q_0p+r_0}\oplus\dots\oplus\grundy{q_cp+r_c}\label{eq:ett}\\
&=(q_0s+\grundy{r_0})\oplus\dots\oplus (q_cs+\grundy{r_0})\label{eq:tva}\\
&=(q_0\oplus\dots\oplus q_c)s+\grundy{r_0}\oplus\dots\oplus\grundy{r_c}\label{eq:tre},
\end{align}
where \eqref{eq:ett} is by Sprague-Grundy theory, \eqref{eq:tva} is by Lemma~\ref{lem:method}, and (\ref{eq:tre}) is by the two constraints.
\end{proof}

Lemma~\ref{cor:prop14} becomes useful in proving that the AP-test implies arithmetic-periodicity by providing invariance of \G-values along arithmetic progressions of heap sizes. Namely the remainder part of the heap sizes in specified options can remain the same, whereas selected parts of the dividend part can be cancelled out, in a similar sense as in Lemma~\ref{lem:h0}. This observation will appear a couple of times in the proof of Theorem~\ref{lem:BigFatLemma}. This theorem will be proved by induction, with a rather technical base case, which will make use of the condition $ c_k\geq 2$ in AP3. We consider a part of this base case in the following lemma.%G

\begin{lemma}\label{lem:3p5p}
  Let $\C=\{ c_1,\ldots, c_k\}$ be a ruleset with $ c_k\geq 2$ such that $\AP{\C}=(p,t)$, with $s=2^t$, $t\in\N_0$.
  Then for each $H_n$ with
  \begin{enumerate}
    \item $n\in [3p+1,4p]$ and for each \G-value $g\in [ 0, 2s-1]$, there is an option $O_n = (h_0,\dots,h_c)$, with $c\geq 2$ such that $\grundy{O_n}=g$,
  \item $n\in [2p+1,3p]$ and for all \G-values $g\in [ 0, s-1]$, there is an option $O_n = (h_0,\dots,h_c)$,  with $c\geq 2$ such that $\grundy{O_n}=g$.
  \end{enumerate}
\end{lemma}
\begin{proof}

\noindent Case 1: Let $H_n$ be such that $n=3p+r_n\in [3p+1,4p]$ and let $g\in [0,2s-1]$. AP3 implies that for each $g\in[0,s-1]$, there is an option $O_n$ such that $\grundy{O_n}=g$.

By the conditions AP1 and AP2, Lemma~\ref{lem:method} gives that $\grundy{n}=3s+\grundy{r_n}$ and hence, by the mex rule, for each $g\in [s,2s-1]$, there is an option $O_n$ such that $\grundy{O_n}=g$. Therefore, if $1\notin \C$, there is nothing to prove. Consequently, it suffices to prove that if $1\in \C$, and $O_n=(h_0,h_1)$ is an option of $n$ obtained by a $1$-cut, then $\grundy{O_n}\notin [s, 2s-1]$. 

    Assume $1\in \C$ and consider the 1-cut $O_n=(h_0, h_1)$. There exist four unique nonnegative integers $q_0,r_0,q_1,r_1$ such that $0 < r_0,r_1\le p$ and $O_n=(q_0p+r_0,q_1p+r_1)$. As $O_n$ is an option of $n$ we have
    $$ (q_0+q_1)p+r_0+r_1=n=3p+r$$
    which gives
    $$ r_0+r_1-r=(3-q_0-q_1)p.$$
    As $0\leq q_0+q_1\leq 3$ and $r_0+r_1\le 2p$, we have on one hand $0\leq r_0+r_1-r<2p$ and on the other hand $r_0+r_1-b\equiv 0 \pmod{p}$. Hence $r_0+r_1-b\in\{0,p\}$. If it equals $0$ then $q_0+q_1=3$, and otherwise $q_0+q_1=2$. Without loss of generality the possible values for $q_0$, $q_1$ and $q_0\oplus q_1$ are summarized in the following table:
         \begin{center}
         \begin{tabular}{|c|c|c|}
           \hline
           $q_0$ & $q_1$ & $q_0\oplus q_1$ \\
           \hline
           \multirow{2}{*}{$0$} & $2$ & $2$ \\
           \cline{2-3}
             & $3$ & $3$ \\
           \hline
           \multirow{2}{*}{$1$} & $1$ & $0$\\
           \cline{2-3}
            & $2$ & $3$ \\
           \hline
         \end{tabular}
         \end{center}
         In particular, $q_0\oplus q_1\ne 1$. And, by Lemma~\ref{cor:prop14}, property (\ref{etoile}) we have that $$\grundy{O_n}=(q_0\oplus q_1)s+\grundy{r_0}\oplus\grundy{r_1}\notin [s, 2s-1],$$ since $s$ is a power of two and $\grundy{r_0},\grundy{r_1} < s$.\\

       \noindent Case 2: Let $n\in [2p+1, 3p]$ and $g\in [0, s-1]$, and let $n'=n+p\in [3p+1,4p]$ and $g'=g+s\in [s, 2s-1]$.

         By the first part of the proof, we know that there is an option of $H_{n'}$, $$O_{n'}=(q_{0,n'}p+r_{0,n'},\dots,q_{m,n'}p+r_{m,n'})$$ such that $m\geq 2$ and $\grundy{O_{n'}}=g'$. Let $N=q_{0,n'}\oplus\dots\oplus q_{m,n'}$, $R=\grundy{r_{0,n'}}\oplus\dots\oplus\grundy{r_{m,n'}}$ and $S=q_{0,n'}+\dots+q_{m,n'}$. We apply Lemma~\ref{cor:prop14} to $O_{n'}$ and get $\grundy{O_{n'}}=Ns+R$. Moreover $N=1$, by $g'\in[s,2s-1]$.

         Define a disjunctive sum of heaps $h=(h_0,\ldots ,h_c)$ by
         $$\begin{array}{lllll}
           h_0&=&r_{0,n'}&&\\
           &&&&\\
           h_j&=&\frac{1}{2}(S-1)p+r_{j,n'}&&\text{ for }
           j=1,2\\
           &&&&\\
           h_j&=&r_{j,n'}&&\text{ for } 3\leq j\leq c
         \end{array}$$
         If $c=2$, ignore the third part. Note that $S-1\ge 0$ is even by Lemma~\ref{lem:parity}. The disjunctive sum $h$ is an option of $H_n$ since
         \begin{align*}
         h_0+\dots+ h_c&=(S-1)p+(r_{0,n'}+\dots+r_{c,n'})\\
         &=n'-p\\
         &=n
         \end{align*}
          Again, since $S-1$ is even,
          \begin{align*}
          \grundy{h} &= R\\
          &=\grundy{O_{n'}}-s\\
          &= g'-s\\
          &= g
          \end{align*}
          Hence, $h=O_n$ is indeed an option of $n$ with $c\geq 2$ and $\grundy{O_n}=g$.
\end{proof}

We can now prove Theorem~\ref{lem:BigFatLemma}: if a partition game verifies the AP-test, then its \G-sequence is purely arithmetic-periodic.

\begin{proof}[Proof of Theorem~\ref{lem:BigFatLemma}]

  Assume $\AP{\C}=(p,t)$, with $2^t=s$. Thus for $0\leq q<4$ where $n=qp+r\in [qp+1,(q+1)p]$, by Lemma~\ref{lem:method}, we have $\grundy{n}=\grundy{qp+r}=qs+\grundy{r}$. 

 We will now prove by induction that for any heap $H_n$, with $n=qp+r\geq 1$, $\AP{\C}=(p,t)$ implies the following two properties (where (A) implies the result, and where (B) is an auxiliary means for the induction). 

  \begin{enumerate}[label=(\Alph*)]
  \item $\grundy{n}=qs+\grundy{r}$ and
  \item for all \G-values $g\in [0, (q-1)s-1]$, there is an option $O_n = (h_0,\dots,h_c)$ such that $c\geq 2$ and $\grundy{O_n} = g$.
  \end{enumerate}

We have a base case.  By Lemma~\ref{lem:method}, (A) holds for all $n\leq 4p$. Moreover, by Lemma \ref{lem:3p5p}, (B) holds for $q=2,3$, and it is trivially true for $q\leq 1$. So (B) holds for $n\le 4p$.

For the induction step, let $n=qp+r>4p$. We will prove the result in two steps: first we show that $\grundy{n}\ge qs+\grundy{r}$ together with $(B)$, and then we show that $\grundy{n}\le qs+\grundy{r}$.

\begin{enumerate}
  \item Case $\grundy{n}\ge qs+\grundy{r}$:
   By induction, the heap of size $n'=n-2p=q'p+r'$ verifies conditions (A) and (B). Consider any \G-value $g<(q'-1)s$. Then, by (B), there is an option $O_{n'}=(q_{0,n'}p+r_{0,n'},\dots,q_{c,n'}p+r_{c,n'})$, with $c\ge 2$ and $\grundy{O_{n'}}=g$. Let $N = q_{0,n'}\oplus\dots\oplus q_{m,n'}$, $S = q_{0,n'}+\dots+q_{c,n'}$ and $R=\grundy{r_{0,n'}}\oplus\dots\oplus \grundy{r_{c,n'}}$. Define $O_{n}=(h_0,\ldots ,h_c)$ by
  $$\begin{array}{lllll}
  h_0 &= &Np+r_{0,n'}&&\\
  &&&&\\
  h_j &=& \frac{1}{2}(S-N+2)p+r_{j,n'} && \text{for } j=1,2\\
  &&&&\\
  h_j &= &r_{m,n'}&& \text{for } 2<j\le c
  \end{array}
  $$
If $c=2$, then omit the third part. This is an option of $H_n$ with at least two cuts since $h_0+\dots +h_c=(2+S)p+r_{0,n'}+\dots+r_{c,n'}$ and its \G-value is $\grundy{O_n}=Ns+R=g$,  by Lemma~\ref{cor:prop14} (since the contributions $(S-N+2)s/2$ for $j =1,2$ cancel out).

  Hence $H_n$ has options to all \G-values in $[0,(q'-1)s-1]$, \textit{i.e.} $\grundy{n}\geq (q-3)s$.\\
  
  We now change $O_n$ into $\hat O_n=(\hat h_0,\ldots ,\hat h_c)$ as follows:
  $$\begin{array}{lllll}
    \hat h_0 &=& q_{0,n'}p+r_{0,n'}+2p &&\\
    &&&&\\
    \hat h_j& =& q_{j,n'}p+r_{j,n'} && \text{for } 0<j\le c
  \end{array}$$
  This is an option of $H_n$ with a $c$-cut since $\hat h_0+\dots +\hat h_m=n'+2p$. Its \G-value is $\grundy{\hat O_n}=(N+2)s+R=g+2s$.

  Hence $H_n$ has options of \G-values in $[2s, (q-1)s-1]$. \\
  
  If $q>4$ then, by putting together the two cases, $H_n$ has options to all \G-values in $[0,(q-1)s-1]$, with $c\ge 2$ and $(B)$ holds.
  Otherwise, if $q=4$, then we take an option $O_{n'}=(q_{0,n'}p+r_{0,n'},\dots, q_{c,n'}p+r_{m,n'})$ of $n'=3p+r=n-p$ with \G-value $g\in [0,s-1]$ and $c\geq 2$, which exists by Lemma \ref{lem:3p5p}. We denote by $S=q_{0,n'}+\dots+	q_{c,n'}$, $N=q_{0,n'}\oplus \dots\oplus q_{c,n'}$ and $R=\grundy{r_{0,n'}}\oplus\dots\oplus\grundy{r_{c,n'}}$, and transform $O_{n'}$ into an option $O_n=(h_0,\dots,h_c)$ by setting
  $$\begin{array}{lllll}
  h_0&=&(N+1)p+r_{0,n'} &&\\
  &&&&\\
  h_j&=& \frac12 (S-N)p+r_{j,n'}&&\text{for } j=1,2\\
  &&&&\\
  h_j&=& r_{j,n'}&& \text{for } 3\leq j\leq c
  \end{array}
  $$
 This is an option of $H_n$ since $h_0+\dots +h_c=(S+1)p+r_{0,n'}+\dots+r_{c,n'}=n'+p=n$ and its \G-value is $\grundy{O_n}=\grundy{O_{n'}}+s=g+s$.

 This gives that for $q=4$, $H_n$ has options obtained by $c$-cuts, $c\geq 2$, to all \G-values in $[0,(q-1)s]=[0,3s]$.  Hence $H_n$ verifies (B) in every case.

  For the remaining part of the proof of (A), let $n'=n-(q-1)p=p+r$ and $g\in [0,s+\grundy{r}-1]$.

  Let  $O_{n'}=(q_{0,n'}p+r_{0,n'},\dots, q_{c,n'}p+r_{c,n'})$ be an option of $H_{n'}$ such that $\grundy{O_{n'}} = g$. It exists, since $H_{n'}$ verifies (B) by induction. Note that as $n'\leq 2p$, if there is a $j$ such that $q_{j,n'}\ne 0$, then it is unique. Hence, without loss of generality, assume that $q_{0,n'}\in\{0,1\}$ and for $j>0$, $q_{j,n'}=0$. Therefore, if $R=\grundy{r_{0,n'}}\oplus\dots\oplus\grundy{r_{m,n'}}$, then $\grundy{O_{n'}}=q_{0,n'}s+R$ by Lemma~\ref{cor:prop14}.

  Define $O_{n}=(h_0,\ldots , h_c)$ by
  $$\begin{array}{lllll}
    h_0&=&(q_{0,n'}+q-1)p+r_{0,n'} && \\
    &&&&\\
    h_j &=& r_{j,n''} & & \text{for } j>0
  \end{array}$$
  This is an option of $H_n$, since $h_0+\dots +h_c = (q_{0,n'}+q-1)p+r_{0,n'}+r_{1,n'}+\dots +r_{c,n'}=n'+(q-1)p=n$. Its \G-value is $\grundy{O_n}=(q_{0,n'}+q-1)s+R=g+(q-1)s$.
  Hence $H_n$ has options to all \G-values in $[(q-1)s, qs+\grundy{r}-1]$. Altogether, $H_n$ has options to all \G-values in $[0, qs+\grundy{r}-1]$.
  
  \item Case $\grundy{n}\le qs+\grundy{r}$: \\
  Assume, as a contradiction, that $O_n=(q_0p+r_0,\dots, q_cp+r_c)$ is an option of $H_n$ with \G-value $\grundy{O_n}=qs+\grundy{r}$, and where as usual $n=qp+r$.  
  Let $N=q_0\oplus\dots\oplus q_c$ and $R=\grundy{r_0}\oplus\dots\oplus \grundy{r_c}$. With this notation, $\grundy{O_n}=Ns+R$, by Lemma~\ref{cor:prop14}. Hence $q=N$ and

\begin{align}\label{eq:R}
\grundy{r}=R,
\end{align}
since $s=2^t$ and $\grundy{r_i},\grundy r <s$, by AP2.

  Define $O_{n'}=(h_0',\ldots ,h'_c)$ by
  $$\begin{array}{lllll}
  h_0'&=&(q-2)p+r_0 &&\\
  &&&&\\
  h_j'&=&r_j, && \text{for } j>1.
  \end{array}$$
  This is an option of $n'=n-2p$, and its \G-value is
  \begin{align}
  \grundy{O_{n'}}&=(q-2)s+R\label{eq:O}\\
  &=qs+\grundy{r}-2s\label{eq:q}\\
  &=\grundy{n'}\label{eq:n'},
  \end{align}
  where \eqref{eq:O} and \eqref{eq:n'} are by Lemma~\ref{cor:prop14}, and \eqref{eq:q} is by \eqref{eq:R},
  a contradiction.\\

  Putting together steps $(1)$ and $(2)$, the heap $H_n$ verifies (A).
    \end{enumerate}

\end{proof}

\subsection{Relaxed conditions on the AP-test}\label{sec:relaxedAP}

We now prove that for some families of games, the conditions AP1 and AP2 of the AP-test imply the condition AP3. We first prove that this is the case if the players are allowed to split a heap into at least one even number and one odd number of heaps. Again, Lemma~\ref{lem:h0} applies in the proof, but the situation gets a little technical, so we write out all details.

\begin{prop}\label{prop:1,2=>3}
  Let $\C$ be a ruleset with $|\C|\ge 2$. If \preG{\C} $(p,t)$-verifies AP1 and AP2, and there are cut numbers $c,c'\in \C$ of different parities with  $2\leq c,c' \leq 2p+1$, then $\AP{\C}=(p,t)$.
\end{prop}
\begin{proof}
Let $s=2^t$. It suffices to prove that $G(\C)$ $(p,t)$-verifies AP3. We prove that for all heap sizes $n\in[3p+1, 4p]$ and all \G-values $g\in[0, s-1]$, there is an option $O_n = (h_0,\dots,h_c)$ such that $c\geq 2$ and $\grundy{O_n} = g$. As usual, write $n=3p+r$ with $0< r \le p$.

  By AP2, for any $g\in [0, s-1]$, there exists a $k\in [1, p]$ such that $\grundy{k} = g$. Let $n' = n-k = 3p+r-k$. The desired cuts will depend on the parity of $n'$.\\

  \noindent Case $n'$ even: let $c\ge 2$ be even, and let $(q', r')$ be the unique ordered pair such that $0 < r' \le c$ and  $n' = cq'+r'$. In particular, $r'$ is even, since $c$ and $n'$ are also even. Moreover $q'>0$ since $c\le 2p< n'$. We define $h=(h_0,\ldots ,h_c)$ by
    $$\begin{array}{lllll}
    h_0&=&k&&\\
    &&&&\\
    h_j&=&q+\frac{1}{2}r &&\text{for  }j=1,2\\
    &&&&\\
    h_j&=&q&&\text{for }3\leq j\leq c
  \end{array}$$
  (If $c=2$, then the third part does not apply.) Then $h=O_n$ is an option of $H_n$ since all $h_j\in\N$, and since $h_0+\dots+h_{c}=k+cq+r=k+n'=n$. Moreover, in the expression $\grundy{h_0}\oplus\dots\oplus\grundy{h_{c}}$, the values  $\grundy{h_1}$ and $\grundy{h_3}$ appear an even number of times, which gives directly $\grundy{O_n}=\grundy{k}=g$.\\

  \noindent Case $n'$ odd: let $c\ge 3$ be odd and let $(q,r)$ be such that $0\leq r<c$, $n' = cq'+r'$. Note that $q'>0$ since $c\le 2p < n'$. As $n'$ and $c$ are odd, either $q'$ is even and $r'$ is odd or vice versa.
    \begin{itemize}
    \item if $q'$ is even and $r'$ is odd, define the option $O_n=(h_0,\ldots ,h_c)$ by:
      $$\begin{array}{lllll}
      h_0 &=& 1+k&&\\
      &&&&\\
      h_j&=&\frac{3}{2}q'+\frac12(r'-1) &&\text{for } j=1,2\\
      &&&&\\
      h_j&=&1 &&\text{for } j=3\\
      &&&&\\
      h_{j}&=&q'&& \text{for } 4\leq j\leq c
    \end{array}
      $$
      If $c = 3$ then we only take the four first heaps.
    The option $O_n$ is an option of $n$ since $h_0+\dots+h_{c}=1+c+3q'+r'-1+1+(c-1-2)q'=1+k+q'_2+r'_2=1+c+n'=n$. In the expression $\grundy{h_0}\oplus\dots\oplus\grundy{h_{c}}$ the terms $\grundy{h_1}$ and $\grundy{h_4}$ appear an even number of times and $\grundy{h_{3}}=0$, and hence $\grundy{O_n}=1+k=g$.
    \item if $q'$ is odd and $r'$ is even, we define the option $O_n$ by:
      $$\begin{array}{lllll}
      h_0&=&1+k&&\\
      &&&&\\
      h_j&=&\frac12(3q'-1)+\frac12r' &&\text{ for } j=1,2\\
      &&&&\\
      h_{j}&=&1&&\text{ for } j=3\\
      &&&&\\
      h_j&=&q' &&\text{ for } 4\leq j\leq c
    \end{array}$$
      it is an option of $n$ since $h_0+\dots+h_{c}=1+k+3q'-1+r+1+(c-3)q'=1+k+cq'+r'=n$. In the expression $\grundy{h_0}\oplus\dots\oplus\grundy{h_{c}}$ the terms $\grundy{h_1}$ and $\grundy{h_4}$ appear an even number of times and $\grundy{h_{3}}=0$. Hence $\grundy{O_n}=g$.
      \end{itemize}
  In each case, there is an option $O_n$ of $H_n$ obtained by a $c$-cut, $c\geq 2$, such that $\grundy{O_n}=g$, \emph{i.e.} \preG{\C} $(p,t)$-verifies AP3, which means that  $\AP{\C}=(p,t)$.
  \end{proof}

If $|\C|=2$ and one possibility is to split a heap into exactly two heaps, and the other is to split into a a given odd number of  heaps, then sometimes AP1 and AP2 imply AP3. 

  \begin{prop}\label{prop:AP3pairs}
  Let $\C = \{1, c\}$ with $ c\ge 4$ even, and suppose that \preG{\C} $(p,t)$-verifies $AP1$ and $AP2$, with $ c \leq p$. If there exist heap sizes $1\le x_1,x_2 \leq p/2$ of different parity, such that $\grundy{x_1} = \grundy{x_2}$, then $\AP{\C}=(p,t)$.
  \end{prop}
  \begin{proof}
We prove that the game \preG{\C} $(p,t)$-verifies AP3, \emph{i.e.}, for any heap $H_n$ with $n\in [3p+1,4p]$ and for any \G-value $g\in [0,2^t-1]$, there exists an option $O_n$, such that $\grundy{O_n} = g$. Since AP2 is $(p,t)$-verified, this can be done by proving that, for all heap sizes $k \in [1, p]$, there exists an option $O_n$ such that $\grundy{O_n} = \grundy{k}$.

  The proof is divided into four cases depending on the parities of $k$ and $n$, and we summarize the relevant $c$-cuts in a table:

   \begin{center}
   {\renewcommand{\arraystretch}{2}%
         \begin{tabular}{|c|c|c|}
          \hline
          $O_n$ & $n$ odd & $n$ even \\
          \hline
          {$k$ odd} &$\left(k, f(n,k),f(n,k),1,\ldots ,1\right)$ & $\left(k, x_1,x_2,f(n,k),f(n,k),1,\ldots ,1\right) $\\
          %\hline
         {$k$ even} &$\left(k,  x_1,x_2, f(n,k),f(n,k),1,\ldots ,1\right)$ & $ \left(k,f(n,k),f(n,k),1,\ldots ,1\right) $\\
         \hline
         \end{tabular}}
         \end{center}
Here $$f(n,k)=\frac{n-k-c}{2}+1$$ if $n$ and $k$ have the same parity, and $$f(n,k)=\frac{n-k- c-x_1-x_2+5}{2}$$ if $n$ and $k$ have different parities and $c\ge 6$. If $c=4$, then by say $0<x_1<x_2$ different parities, we consider instead the option $O_n=\left(k, x_1,x_2-1,f(n,k),f(n,k)\right)$. It is straightforward to justify in each case that $O_n$ is an option, and then apply Lemma~\ref{lem:h0}.

  Hence, for all heap sizes $k \in [1,p]$, there exists an option of $n$ with the same \G-value. This implies that the condition AP3 is $(p,t)$-verified, and thus that the AP-test is $(p,t)$-verified for \preG{\C}.
  \end{proof}

We now prove that the conditions of Proposition~\ref{prop:AP3pairs} are verified for all those games as long as the period is lower bounded by $4 c+3\leq p$. 

\begin{corollary}\label{cor:AP3pairs}
Let $\C=\{1, c\}$ with $ c\ge 4$ even. If \preG{\C} $(p,t)$-verifies AP1 and AP2 for some $p\geq 4 c+3$, then $\AP{\C}=(p,t)$.
\end{corollary}
\begin{proof}
By Proposition~\ref{prop:AP3pairs}, it suffices to prove that there exists $x_1,x_2<p/2$ of different parities, such that $\grundy{x_1}=\grundy{x_2}$.

Note that $\grundy{2}=1$, since the only option is $(1,1)$ which has $\G$-value $0$. Hence we can assume $x_2=2$.

We claim that we can choose $x_1=2 c+1$ to get $\grundy{x_1}=\grundy{2}=1$. In order to do that, we prove that the beginning of the \G-sequence of the game \preG{\C} is $(0,1)^{ c/2}$ and the following $ c$ values are different from $1$ and $0$, and the $2 c+1$-th value is $1$. 
If $k\leq  c$, then $\grundy{k} =(k-1)\pmod 2$, by Proposition~\ref{prop:1odd}.

Now, let $k\in [c+1,2 c]$. If $k$ is odd, then $k$ admits the $1$-cut option $(k- c, c)$ of Grundy value $1$ since $ c$ is even, and the $ c$-cut option $(k- c,1,\dots,1)$ of Grundy value $0$. If $k$ is even, it admits the $ c$-cut option $(k- c,1,\ldots,1)$ of Grundy value $1$, and the $1$-cut option $(k/2,k/2)$ of Grundy value $0$. It thus implies that $\grundy{k}>1$.

Finally, we prove $\grundy{2 c+1}=1$. We now set $k=2 c +1$.

From $k$, one can reach the value $0$ by the option $(1,2,\ldots,2)$ obtained by a $ c$-cut. All the $1$-cuts $(i_0,i_1)$ are such that without loss of generality $i_0> c$ and $i_1\leq c$, so $\grundy{i_0,i_1}\neq 1$ since $\grundy{i_1}< 2$ and $\grundy{i_2}\geq 2$.

It suffices to prove that there is no $ c$-cut $O_k=(i_0,\ldots,i_ c)$ with $\grundy{O_k}=1$. If there is some $j$ such that $i_j> c$, then it is unique, and as before $\grundy{i_j}\ge 2$  implies $\grundy{O_k}\geq 2$. Hence, assume, for all $j$, $i_j\leq  c$. If the \G-value of the $c$-cut were 1, then we must have an odd number of even heap sizes, and hence an even number of odd heap sizes, since $c$ is even. But $2c+1$ is odd, so the $c$-cut cannot sum up correctly. 
Therefore, $\grundy{2 c+1}=1$. Moreover, $2 c+1<p/2$ since $4 c+3\leq p$. 
\end{proof}

\subsection{Applications of the AP-test}

Table~\ref{tab:someGames} summarizes applications of the AP-test for some partitioning games (the games already solved in Section~\ref{sec:psg} are not in the table). All the games in this list satisfy the test and hence are proved to be purely arithmetic-periodic. More specifically, Corollary~\ref{cor:AP3pairs} has been applied to the games $\{1,4\}$, $\{1,6\}$,  $\{1,8\}$, and $\{1,10\}$. We note that for games of the form $\{1, c\}$, there seem to be rather long periods depending on $ c$, with always the same saltus. We wonder whether this regularity holds for higher values of $ c$. 

\begin{problem}
Given $ c\geq 2$, the game \preG{\C} with $\C=\{1,2 c\}$ is arithmetic-periodic of length $12 c$ and saltus $8$.
\end{problem}

Surprisingly, when one adjoins new values to the games $\{1,2 c\}$ (with $ c\geq 2$),  the period and saltus may change significantly. This is for example the case for the game $\{1,4\}$, which has period 24 and saltus~8, but $\{1,3,4\}$ has period 4 and saltus 2, while $\{1,4,9\}$ has period 40 and saltus 16.

The cases where $1,2 \in \C$ but $3 \notin \C$ remain the hardest to understand. If Table~\ref{tab:someGames} suggests a purely arithmetic-periodic behavior when $|\C|\geq 3$, we did not detect any general pattern, depending on $\C$. For example, when $|\C|=3$, the games $\{1,2,4\}$ and $\{1,2,6\}$ have identical Sprague-Grundy sequences, whereas $\{1,2,5\}$ and $\{1,2,7\}$ are more singular. Even worse, the games $\{1,2,8\}$ and $\{1,2,7,8\}$ seem to be ultimately arithmetic-periodic with saltus $8$ and a preperiod of positive length equal to $6$ (which is not the case of the other sequences we computed).\footnote{These are the only examples we found with a non-0 preperiod.} A proof of this result remains to be done, as the AP-test does not cover ultimate behavior (and does not seem to be easily adapted for them).

\begin{table}[!h]
	\begin{bigcenter}
\begin{tabular}{|c|c|} \hline
{\sc cut}-set & Sprague-Grundy sequence \\ \hline
$\{1,4\} \cup K$ & \multirow{2}{*}{$((0,1)^2(2,3)^2,1,4,5,4,(3,2)^2(4,5)^2(6,7)^2)~(+8)$} \\
with $K \subseteq \{6,8,10\}$ & \\ \hline
$\{1,6\} \cup K$ & \multirow{2}{*}{$((0,1)^3(2,3)^3,1,4,(5,4)^2(3,2)^3(4,5)^3(6,7)^3)~(+8)$} \\
with $K \subseteq \{8,10\}$& \\ \hline
$\{1,8\}$ & $((0,1)^4(2,3)^4,1,4,(5,4)^3(3,2)^4(4,5)^4(6,7)^4)~(+8)$ \\ \hline
$\{1,10\}$ & $((0,1)^5(2,3)^5,1,4,(5,4)^4(3,2)^5(4,5)^5(6,7)^5)~(+8)$ \\ \hline
$\{1,4\} \cup K$ & \multirow{2}{*}{$(0,1)^2~(+2)$} \\
with $K \subseteq \{3,5,6,7,8\}$ and 3, 5 or 7 $\in K$ & \\ \hline
$\{1,6\} \cup K$ & \multirow{2}{*}{$(0,1)^3~(+2)$} \\
with $K \subseteq \{3,5,7,8\}$ and 3, 5 or 7 $\in K$ & \\ \hline
$\{1,8\} \cup K$ & \multirow{2}{*}{$(0,1)^4~(+2)$} \\
with $K \subseteq \{3,5,7\}, K \neq \emptyset$ & \\ \hline
$\{1,2,4\} \cup K, \{1,2,6\} \cup K'$ & \multirow{2}{*}{$(0,1,2,3,1,4,3,2,4,5,6,7)~(+8)$} \\
with $K \subseteq \{6,7,8\}, K' \subseteq \{7,8\}$ & \\ \hline
$\{1,2,5\} \cup K$ & \multirow{2}{*}{$(0,1,2,3,1,4,3,6,4,5,6,7)~(+8)$} \\
with $K \subseteq \{4,6,7,8\}$ & \\ \hline
$\{1,2,7\}$ & $(0,1,2,3,1,4,3,2,4,5,6,7,8,9,7,6,9,8,11,10,12,13,10,11,13,12,15,14)~(+16)$ \\ \hline
\multirow{2}{*}{$\{1,4,9\}$} & $(0,1,0,1,2,3,2,3,1,4,5,4,3,6,7,6,4,5,8,9,6,7,10,11,9,$ \\
& $8,9,12,11,10,11,14,12,13,12,13,14,15,14,15)~(+16)$ \\ \hline
%$\{1,2,8\}, \{1,2,7,8\}$ & $(0,1,2,3,1,4)~(3,2,4,5,6,7,8,9,7,11,9,8)~(+8)$ \\ \hline
\end{tabular}
\end{bigcenter}
\caption{Some partition games for which the purely arithmetic-periodicity is proved with the AP-test.}
\label{tab:someGames}
\end{table}

\section{A review of {\sc take-and-break} in relation with {\sc cut}}\label{sec:background}
We review some history of theory intertwining {\sc subtraction} with {\sc cut}, and such games are gathered under the umbrella {\sc take-and-break}. (Then in Section~\ref{sec:conper}, we return to conclude our findings in this broader perspective.)

\subsection{{\sc take-and-break}}
{\sc take-and-break} \cite{WW} is a large class of 2-player impartial combinatorial games with alternating play. A game position is represented by a finite multiset of heaps. A move consists in choosing a single heap, removing some beans, and possibly splitting the remaining heap into several heaps. If splitting is not allowed, we have (pure) {\sc subtraction}. In this case, the rules are given by a set $\S$ of positive integers which specifies the number of tokens that can be removed from the heap. When the heap may be split, the rulesets are often given by a code that specifies how many tokens can be removed and the number of heaps that the remainder of the heap can be split into.

Let us review some of the history of {\sc take-and-break}. The origin to the line of research was the invention of a fairy game of {\sc chess} in the 1930s, called {\sc dawson's chess} \cite{D}, named by its inventor, played by using only the pawns of {\sc chess} and with mandatory captures. It was soon remarked that the game is equivalent to a game on  heaps of counters, where a player may remove (i) a single counter if by removing it no heap remains, (ii) two counters if no heap or a single heap remain (iii) three counters if no heap, one heap or two heaps remain. In the normal play convention it was discovered that the solution has a regular behavior (ultimately periodic) \cite{GS}. The elegance of the solution of {\sc dawson's chess}, together with the similar game of {\sc kayles} \cite{GS} (remove one or two counters anywhere), gave rise to classifications of games on heaps of counters \cite{GS}.

For example, the family of games for which a heap can be split into at most two heaps is  nowadays called {\sc octal} games. This name is due to an explicit way to express any ruleset with an octal code $\mathbf{d_0.d_1\cdots d_k}$ with $d_i$ an integer, $0 \leq d_i \leq 7$ for $1 \leq i \leq k$. More precisely, each value $d_i$ with $i>0$ is encoded in binary as $d_i=2^0e_0+2^1e_1+2^2e_2$, where each $e_j\in\{0,1\}$. The ruleset allows to remove $i$ tokens from a heap and split the remainder into $j$ non-empty heaps if and only if $e_j=1$. So, for example {\sc dawson's chess} is the {\sc octal} game $\mathbf {0.137}$, and {\sc kayles} is the game $\mathbf{0.77}$. 
Observe that, for {\sc octal}, the value $d_0$ equals $0$ or $4$ according to whether it is allowed (value $4$) or not (value $0$) to split a heap without removing any token. Variants of {\sc octal}, where the ruleset allows to split a heap without removing any token, have been considered in the literature, starting from {\sc grundy's game} in 1939~\cite{G}. An extra condition has been imposed to {\sc grundy's game}, to avoid trivialities, namely the outcome of a split must result in heaps of non-equal sizes. Hence, {\sc grundy's game} is not in {\sc octal}.

The notation for {\sc octal} has been extended to define an arbitrary instance of (pure) {\sc take-and-break}, by giving any integer value to each $d_i$. According to the binary decomposition of $d_i=2^0e_0+2^1e_1+2^2e_2+2^3e_3\ldots$, each value $e_j\in\{0,1\}$ indicates whether it is allowed to remove $i$ tokens and split the remainder into $j$ non-empty heaps. As an example, in the family of {\sc hexadecimal} games, each $d_i$ is in $\{0,\ldots,2^4-1\}$ (generally encoded by an hexadecimal notation), indicating that a heap can be split into at most three heaps.

We relate the current work to {\sc take-and-break} by considering particular rulesets, without `take'/removal.
To our knowledge, the full class of `break-without-take' has not yet been studied, at least not anywhere near the extent of the folklore results and conjectures on pure subtraction games, and other {\sc octal} games, which were researched extensively through the last couple of decades, e.g. \cite{F, WW}.

\subsection{Regularities in {\sc take-and-break} games}\label{sec:regtbg}

Given a {\sc take-and-break} game, its $\mathcal{G}$-sequence is the sequence $\grundy{1},\grundy{2},\grundy{3}, \ldots$. Finding regularities in $\mathcal{G}$-sequences is a natural objective as it may lead to polynomial-time algorithms that compute the $\mathcal{G}$-values of the game. In particular, periodic behaviors are often observed. A game is said to be {\em ultimately periodic} with period $p$ and preperiod $n_0$ if there exist $n_0$ and $p$ such that $\grundy{n+p}=\grundy{n}$ for all $n \geq n_0$. {\em Pure periodic} games are those for which there is no preperiod; \emph{i.e.} $n_0=0$.

For example, it is well known (see \cite{Siegel}, Chapter 4, Theorem 2.4 and 2.5) that all finite {\sc subtraction} games are ultimately periodic, and a periodicity test has been shown in~\cite{LIP} (see Chapter 7, Corollary 7.34). For {\sc octal}, the behavior of the $\mathcal{G}$-sequences is far from fully understood, even for very short (finite) rule descriptions. Guy famously asked whether every finite {\sc octal} game is ultimately periodic. Nowadays, one might lean towards that this is not the case, but it seems to be a very hard problem to understand fully. For example, there are 79 non-trivial {\sc octal} games with at most 3 octal-digits among which only 14 have been solved \cite{F}. So, in this sense, {\sc octal} seems perhaps more complicated than {\sc cut}\footnote{Recall that {\sc cut} is necessarily invariant, but octal games are typically non-invariant, by setting for example the $2^0$ component to 0 for some digit.}.

Out of the games that were proved to satisfy this conjecture, one can mention $\mathbf{0.106}$, $\mathbf{0.165}$ and $\mathbf{0.454}$. In some cases, the values of the period and the preperiod are huge (e.g. $\mathbf{0.454}$ has a period of $60620715$ and a preperiod of $160949019$). On his webpage~\cite{F}, Flammenkamp maintains a list of octal games with known and unknown periodicities;  the game 0.106 has period length  328226140474 and a preperiod of length 465384263797.

J. P. Grossman \cite{officer} recently contributed to some ``parallelization strategies'' and other optimizations that accelerate the computation of Sprague-Grundy values of the game $\mathbf{0.6}$, also called {\sc officers}, by nearly 60 times of the previous records. His  implementation computed over 5 million values per second, a total of more than 140 trillion values over a course of 18 months, without finding any period. And note that {\sc officers} is a single digit octal game (!). As was shown in Winning Ways~\cite{WW}, to prove periodicity it suffices to establish that an octal game contains only finitely many \G-values. But the problem is that extensive computations reveal the ``sparse value phenomenon'', and to our knowledge there is no theory yet to handle them. Grossman's approach exploits the fact that sparse values are extremely non-dense, and correctly used, this can dramatically speed up computation, searching for possible regularity (and his method can be used for other games as well).
{\sc officers} is the only single digit octal game for which Richard K. Guy's original question/conjecture remains unanswered. Somewhat surprisingly, the invariant ditto, the game $\mathbf{0.7}$ where, in addition, a heap with a single bean may be removed, reverts to ``she-loves-me-she-loves-me-not''.

As explained in \cite{Hexa}, some hexadecimal games also satisfy these properties of normal periodicity (e.g. $\mathbf{0.B3}$, $\mathbf{0.33F}$). In addition, other types of behavior have been exhibited for hexadecimal games, namely {\em arithmetic-periodicity}. A take-and-break game is {\em ultimately arithmetic-periodic} with period $p$, saltus $s$, and preperiod $n_0$ if there exist three integers $n_0$, $p$ and $s$ such that its $\mathcal{G}$-sequence satisfies $\grundy{n+p}=\grundy{n}+s$ for all $n \geq n_0$; if $n_0=0$, the sequence is \emph{purely arithmetic-periodic}.

If (strict) arithmetic-periodicity never occurs in {\sc octal} \cite{austin}, it makes sense in the context of {\sc hexadecimal}, where the Sprague-Grundy values may not be bounded. For example, the games $\mathbf{0.13FF}$ or $\mathbf{0.9B}$ are proved to be ultimately arithmetic-periodic with period $7$ and saltus $4$ in \cite{Hexa}. Note that normal and arithmetic-periodicities are not the only kinds of regularities that have been detected in hexadecimal games. In~\cite{LIP}, the game $\mathbf{0.205200C}$ is said to be {\em sapp-regular}, which means that the $\mathcal{G}$-sequence is an interlacing of two periodic subsequences with an arithmetic-periodic one; this behavior occurs also in variants of octal games with pass moves \cite{pass}. In \cite{Hexa}, the game $\mathbf{0.123456789}$ satisfies $\grundy{2m-1}=\grundy{2m}=m-1$, except $\grundy{2^k+6}=2^k-1$. In \cite{ruler}, Grossman and Nowakowski use the notion of {\em ruler regularity} that arises in the hexadecimal games $\mathbf{0.20\cdots 48}$ with an odd number of intermediate 0s; roughly speaking, it corresponds to a kind of arithmetic-periodic sequence where new terms are regularly introduced that double the length of the apparent period.

For a better understanding of {\sc take-and-break}, the question of how to detect a possible regularity using just a small number of computations is paramount. For example, the folklore {\sc subtraction} periodicity theorem (see for example Chapter 4 of~\cite{Siegel}) relies on that, for a given finite {\sc subtraction} game $S$, it suffices to find a repetition of $\max(S)$ consecutive Sprague-Grundy values, to establish ultimate periodicity. Concerning {\sc octal}, there is a similar result that has been extensively used to prove the ultimate periodicity of some $\mathcal{G}$-sequences~\cite{Siegel}:

\begin{theorem}[{\sc octal} Periodicity Test]\label{thm:octal}
Let $G$ be an octal game $\mathbf{d_0.d_1\cdots d_k}$ of finite length $k$. If there exist $n_0\geq 1$ and $p\geq 1$ such that
$$
\grundy{n+p}=\grundy{n} ~\forall n\in\{n_0, \ldots , 2n_0+p+k-1\},
$$
then $G$ is ultimately periodic with period $p$ and preperiod $n_0$.
\end{theorem}

Such kind of testing properties have also been considered for {\sc hexadecimal}. In \cite{austin}, Austin yields a first set of conditions to guarantee the ultimate arithmetic-periodicity of {\sc hexadecimal} with a saltus equal to a power of $2$. An extension was later given by Howse and Nowakowski \cite{Hexa} for {\sc hexadecimal} having an arbitrary saltus. In both cases, several types of computations must be done. In particular, the arithmetic-periodicity must be checked on a range of values much larger than in Theorem~\ref{thm:octal} (at least seven times the expected period).

\subsection{The partition game {\sc cut} as a pure-break game}

The partition game {\sc cut} is a pure-break game in two senses. The first meaning is that they have no special rule which restricts the type of breaking possible, in contrast with for example {\sc grundy's game} (a split must result in unequal heap sizes), {\sc couples are forever} (any split into two parts is possible, except if the heap is $H_2$). Secondly there is no ``taking'', \emph{i.e.} no subtraction of tokens is possible at any stage of play.

As noted already in going from {\sc octal} to {\sc hexadecimal}, allowing a heap to be split into three parts may significantly change the behavior of the $\mathcal{G}$-sequence. In the current paper, we have explored how the $\mathcal{G}$-sequences behave when increasing the number of possible splits of a heap. As one might expect that the complexity of this generalization also increases accordingly, we have chosen to focus on the arguably most natural class of break games, namely \emph{break-without-take}, \emph{i.e.} games where it is never allowed to remove tokens from the disjunctive sum of heaps. The total number of beans in the heaps remain the same until the end of play. If the game starts with a large number $n$, then the game proceeds by finding finer and finer partitions of this number, until no more refinement is possible. %The sum of the components remains the same throughout each stage of play.

Theorem~\ref{lem:ericslemma} means that partition games are somehow closer to {\sc hexadecimal} than {\sc octal}. This is not so surprising, by the possibility of an extra cut in {\sc hexadecimal}. Moreover, we note that partition games form particular instances of {\sc take-and-break}.

\begin{prop}\label{prop:equivalence}
The partition game \preG{\C} played on a heap of size $n$ is equivalent to {\sc take-and-break} $\mathbf{0.d_1d_2\cdots }$ played on a heap of size $n-1$, where $d_c=2^{c+2}-1$ if $c\in\C$, and $d_c=0$ otherwise. That is, decompose $d_0$ in binary, so that $d_0=\sum e_i2^{i+1}$.  Then, for all $n>0$, $\mathbf{d_0.00}\cdots(H_n) = \mathbf{0.d_1d_2\cdots}(H_{n-1})$, where $d_i=2^{i+2}-1$ if $e_i=1$ and where $d_i=0$ if $e_i=0$.
\end{prop}

\begin{proof}
Interpret the game \preG{\C} played on a heap of size $n$ as a game on a path $P_n$ of size $n$ as follows. A $k$-cut consists in removing $k$ edges from the path, leaving $k+1$ non-empty smaller paths. Now consider the line graph of $P_n$ (\emph{i.e.} the graph where edges are replaced by vertices, and two vertices are adjacent if the corresponding edges where incident in the initial graph). Clearly, this line graph is a path of size $n-1$. A $k$-cut in $P_n$ thus consists in removing $k$ vertices from the line graph. Hence \preG{C} is equivalent to a take-and-break game $\mathbf{0.d_1d_2\ldots }$ on the path of size $n-1$. As there is no constraint on the edges removed from $P_n$, there is no constraint on the corresponding vertices in the line graph. Hence, if $k\in \C$, the code of each $d_k$ is the maximal possible value, $2^0+2^1+\cdots+2^{k+1}=2^{k+2}-1$, \emph{i.e.} remove $k$ vertices and leave any number of heaps $\le k+1$.
\end{proof}

Proposition~\ref{prop:equivalence} shows that $\preG{\{1,2\}}$ is equivalent to {\sc hexadecimal} $\mathbf{0.7F}$. Applied to the game {\sc couples-are-forever}, the same analysis proves that it is equivalent to {\sc octal} $\mathbf{0.6}$, \emph{i.e.} {\sc officers}; moreover a similar argument shows that {\sc dawson's chess} has two equivalent representations $\mathbf{0.137}$ and $\mathbf{0.4}$ (with two additional tokens), and so on.

This discussion concerned several examples of extremely short {\sc take-and-break} rulesets with great complexity. By relating to the case of \cut\ one might guess that in general, the complexity of the $\mathcal{G}$-sequence does not increase with $\max(\C)$. Indeed in all case of analysis in this paper, except one, namely $\C = \{1,2\}$, the $\mathcal{G}$-sequence is either periodic or arithmetic-periodic. Let us conclude the paper, with related open questions.

\section{Conclusion and perspectives}\label{sec:conper}

The results from Sections~\ref{sec:psg} and \ref{sec:aptsg} are summarized in Table~\ref{tab:recap}. The games are partitioned into three families: those for which periodicity or arithmetic-periodicity is proved, and those for which we used two or three conditions of the AP-test, respectively.

\begin{table}[!h]
\centering
\begin{tabular}{|c|c|c|c|}
\hline
& The ruleset $\C$ & Sequence & Theorem  \\ \hline
\multirow{4}{*}{Solved}
& $\{1, c_2,\ldots \}$ ($ c_i$ odd) & $(0,1)~(+0)$ & Proposition~\ref{prop:1odd} \\
& $\{ c_1,\ldots \}$ ($ c_1 > 1$) & $(0)^{ c_1} ~(+1)$ & Proposition~\ref{lem:k*} \\
& $\{1,2,3, c_4,\ldots \}$ & $(0) ~(+1)$ & Proposition~\ref{lem:123} \\
& $\{1,3,2k\}$ ($k \geq 1$) & $(0,1)^ c ~(+2)$ & Proposition~\ref{lem:132k} \\ \hline
\multirow{2}{*}{AP1 and AP2} & $\{1,2 c,2 c'+1, c_1,\ldots, c_k\}$ & & Proposition~\ref{prop:1,2=>3} \\
& $\{1,2 c\}$ ($ c \geq 2$) & & Corollary~\ref{cor:AP3pairs} \\ \hline
AP1, AP2 and AP3 & $\{1, c_1,\ldots, c_k\}$ ($ c_i$ even, $k \geq 1$) & & Theorem~\ref{lem:BigFatLemma} \\ \hline
\end{tabular}
\caption{Some partition games.}
\label{tab:recap}
\end{table}

Among the families that are not solved, all of our computations on particular examples have shown ultimate arithmetic-periodic behaviors, except one, namely \preG{\{1,2\}}. This game has a Sprague-Grundy sequence with a lot of visible regularity but some surprising `drop-out' values, as shown in Figure~\ref{fig:12EstRelou} (in the second picture the drop-outs have been highlighted).

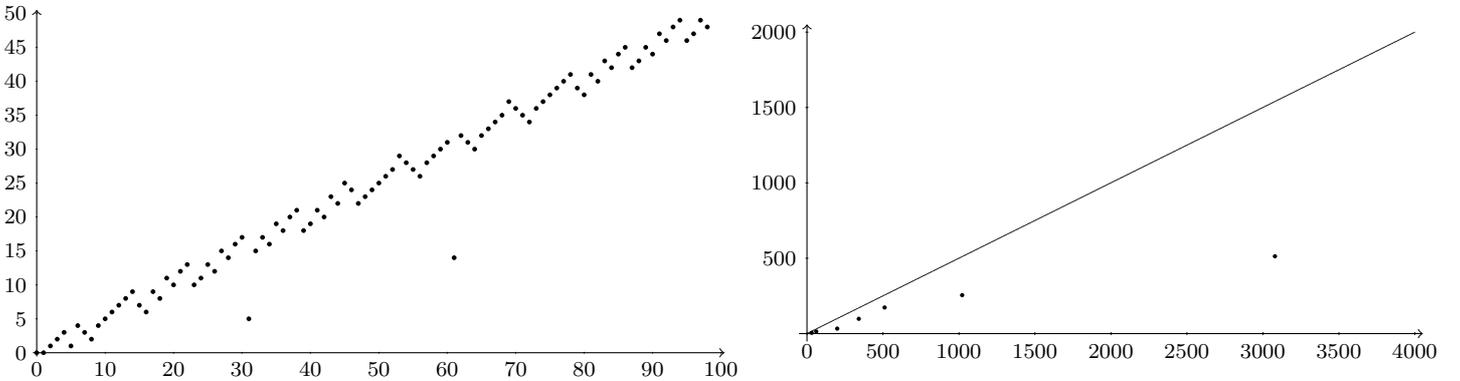
\begin{figure}[!h]
\Wider[3.5cm]{
\begin{minipage}{0.5\linewidth}
\begin{center}
\begin{tikzpicture}[line cap=round,line join=round,scale=0.09]
  \def\limxs{100}
  \def\limxi{0}
  \def\limys{50}
  \def\limyi{0}
  \coordinate (lims) at (\limxs+0.5,\limys+0.5);
  \coordinate (limi) at (\limxi-0.5,\limyi-0.5);
\draw[->,color=black] (\limxi-0.5,0) -- (\limxs+0.5,0);
\foreach \x in {0,10,...,\limxs}
\draw[shift={(\x,0)},color=black] (0pt,2pt) -- (0pt,-2pt) node[below] {\footnotesize $\x$};
\draw[->,color=black] (0,\limyi-0.5)--(0,\limys+0.5);
\foreach \y in {0,5,...,\limys}
\draw[shift={(0,\y)},color=black] (2pt,0pt) -- (-2pt,0pt) node[left] {\footnotesize $\y$};

\clip(limi) rectangle (lims);
\fill (0,0) circle (10pt);
\fill (1,0) circle (10pt);
\fill (2,1) circle (10pt);
\fill (3,2) circle (10pt);
\fill (4,3) circle (10pt);
\fill (5,1) circle (10pt);
\fill (6,4) circle (10pt);
\fill (7,3) circle (10pt);
\fill (8,2) circle (10pt);
\fill (9,4) circle (10pt);
\fill (10,5) circle (10pt);
\fill (11,6) circle (10pt);
\fill (12,7) circle (10pt);
\fill (13,8) circle (10pt);
\fill (14,9) circle (10pt);
\fill (15,7) circle (10pt);
\fill (16,6) circle (10pt);
\fill (17,9) circle (10pt);
\fill (18,8) circle (10pt);
\fill (19,11) circle (10pt);
\fill (20,10) circle (10pt);
\fill (21,12) circle (10pt);
\fill (22,13) circle (10pt);
\fill (23,10) circle (10pt);
\fill (24,11) circle (10pt);
\fill (25,13) circle (10pt);
\fill (26,12) circle (10pt);
\fill (27,15) circle (10pt);
\fill (28,14) circle (10pt);
\fill (29,16) circle (10pt);
\fill (30,17) circle (10pt);
\fill (31,5) circle (10pt);
\fill (32,15) circle (10pt);
\fill (33,17) circle (10pt);
\fill (34,16) circle (10pt);
\fill (35,19) circle (10pt);
\fill (36,18) circle (10pt);
\fill (37,20) circle (10pt);
\fill (38,21) circle (10pt);
\fill (39,18) circle (10pt);
\fill (40,19) circle (10pt);
\fill (41,21) circle (10pt);
\fill (42,20) circle (10pt);
\fill (43,23) circle (10pt);
\fill (44,22) circle (10pt);
\fill (45,25) circle (10pt);
\fill (46,24) circle (10pt);
\fill (47,22) circle (10pt);
\fill (48,23) circle (10pt);
\fill (49,24) circle (10pt);
\fill (50,25) circle (10pt);
\fill (51,26) circle (10pt);
\fill (52,27) circle (10pt);
\fill (53,29) circle (10pt);
\fill (54,28) circle (10pt);
\fill (55,27) circle (10pt);
\fill (56,26) circle (10pt);
\fill (57,28) circle (10pt);
\fill (58,29) circle (10pt);
\fill (59,30) circle (10pt);
\fill (60,31) circle (10pt);
\fill (61,14) circle (10pt);
\fill (62,32) circle (10pt);
\fill (63,31) circle (10pt);
\fill (64,30) circle (10pt);
\fill (65,32) circle (10pt);
\fill (66,33) circle (10pt);
\fill (67,34) circle (10pt);
\fill (68,35) circle (10pt);
\fill (69,37) circle (10pt);
\fill (70,36) circle (10pt);
\fill (71,35) circle (10pt);
\fill (72,34) circle (10pt);
\fill (73,36) circle (10pt);
\fill (74,37) circle (10pt);
\fill (75,38) circle (10pt);
\fill (76,39) circle (10pt);
\fill (77,40) circle (10pt);
\fill (78,41) circle (10pt);
\fill (79,39) circle (10pt);
\fill (80,38) circle (10pt);
\fill (81,41) circle (10pt);
\fill (82,40) circle (10pt);
\fill (83,43) circle (10pt);
\fill (84,42) circle (10pt);
\fill (85,44) circle (10pt);
\fill (86,45) circle (10pt);
\fill (87,42) circle (10pt);
\fill (88,43) circle (10pt);
\fill (89,45) circle (10pt);
\fill (90,44) circle (10pt);
\fill (91,47) circle (10pt);
\fill (92,46) circle (10pt);
\fill (93,48) circle (10pt);
\fill (94,49) circle (10pt);
\fill (95,46) circle (10pt);
\fill (96,47) circle (10pt);
\fill (97,49) circle (10pt);
\fill (98,48) circle (10pt);
\fill (99,51) circle (10pt);
\end{tikzpicture}
\end{center}
\end{minipage}
\begin{minipage}{0.45\linewidth}
\begin{center}
\include{12_long}
\end{center}
\end{minipage}
}
\caption{The Sprague-Grundy sequence of  \preG{\{1,2\}} for $n\leq 100$ and $n\leq 4000$.}\label{fig:12EstRelou}
\end{figure}

In view of our computations, we propose the following problem. 

\begin{problem}
Is it true that every game \preG{\C} with $\C\neq \{1,2\}$ has a Sprague-Grundy sequence that is either ultimately periodic or ultimately arithmetic-periodic?
\end{problem}

A first step to understand this problem would be to justify (or disprove) the following conjecture.

\begin{conjecture}
Every instance \C\ of \cut\ for which $\{1,2\}\not\subset \C$ is (ultimately) arithmetic-periodic.
\end{conjecture}

For some games, the conjecture holds but where a non-computer justification becomes at the best tedious (e.g. $\C=\{1,2,7\}$). A general proof is probably hard to obtain, and meanwhile this motivates the testing conditions. The AP-test is a rather short computation to justify arithmetic-periodicity of a game; for even more simplicity, we wonder whether the condition AP3 could be removed from the test. 

\begin{problem}
Do AP1 and AP2 imply AP3?
\end{problem}

Note that removing AP3 from the AP-test would not have much impact on the computational complexity of the test, but this would make the AP-test closer to its counterparts for other take-and-break games.

Concerning the ultimate property of the sequence, we have seen that in almost all considered cases, there is no preperiod. However, the case $\C=\{1,2,8\}$ suggests that some games have one. Therefore, the question of finding an equivalent of the AP-test for these games makes sense.

\begin{problem}
Find a generalization of the AP-test to prove ultimate arithmetic-periodicity of partition games.
\end{problem}

As indicated by Figure~\ref{fig:12EstRelou}, the game \preG{\{1,2\}} leaves a couple of questions open. 
What is the number of occurrences of each Sprague-Grundy value of $\preG{\{1,2\}}$? We already know from Theorem~\ref{lem:ericslemma} that every Sprague-Grundy value appears at most twice in the sequence of $\preG{\{1,2\}}$ (apply Theorem~\ref{lem:ericslemma} with $c=2$).

\begin{problem}
Does each Sprague-Grundy value appear at least once in the sequence of \preG{\{1,2\}}? More precisely, does each \G-value appear exactly twice in the sequence of \preG{\{1,2\}}?
\end{problem}

\begin{problem}
What is the behavior of the Sprague-Grundy sequence of \preG{\{1,2\}}?
\end{problem}

Recall that from Proposition~\ref{prop:equivalence}, this game is equivalent to {\sc hexadecimal} $\mathbf{0.7F}$ (that is also claimed unsolved), with rules: remove one token and leave no heap, remove one token and leave one heap, remove one token and leave two heaps, remove two tokens and leave no heap, remove two tokens and leave one heap, remove two tokens and leave two heaps, or remove two tokens and leave three heaps. The {\sc cut} description is here more concise, (without removing any token) split your heap in two or three parts,  and with equivalent behavior.

On the other hand a full generalization of hexadecimal games includes as a subclass {\sc cut}, but we feel perhaps the general class is too large, and further classification of partition games appears a more approachable (sub-)goal. Another  interesting direction would be to pursue Grundy's idea to disallow symmetric partitions (but with no removal), but we have not yet looked into that.

\section*{Acknowledgement}
We would like to thank R. J. Nowakowski for interesting discussions regarding the connection of {\sc cut} with other instances of {\sc take-and-break}.

\end{document}